\documentclass[11pt,reqno]{amsart}

\usepackage[T1]{fontenc}
\usepackage[a4paper,margin=2.8cm]{geometry}
\usepackage{amsmath,amssymb,amsfonts,amsthm,mathtools}
\usepackage{enumitem}
\usepackage{microtype}
\usepackage{booktabs}
\usepackage[hidelinks,bookmarksdepth=3]{hyperref}
\usepackage[nameinlink,noabbrev]{cleveref}

\numberwithin{equation}{section}

\theoremstyle{plain}
\newtheorem{theorem}{Theorem}[section]
\newtheorem{lemma}[theorem]{Lemma}
\newtheorem{corollary}[theorem]{Corollary}

\theoremstyle{definition}
\newtheorem{definition}[theorem]{Definition}

\newcommand{\Prb}{\mathbb P}
\newcommand{\E}{\mathbb E}
\newcommand{\calA}{\mathcal A}
\newcommand{\calD}{\mathcal D}
\newcommand{\calM}{\mathcal M}
\newcommand{\calR}{\mathcal R}
\newcommand{\calT}{\mathcal T}

\newcommand{\supp}{\operatorname{supp}}
\newcommand{\Bin}{\operatorname{Bin}}
\newcommand{\dist}{\operatorname{dist}}
\newcommand{\eps}{\varepsilon}
\newcommand{\one}{\mathbf 1}

\usepackage[
  style=numeric,
  backend=biber,
  maxbibnames=99,
  minbibnames=99,
  maxcitenames=99,
  mincitenames=99
]{biblatex}
\AtBeginDocument{
  \setlength{\abovedisplayskip}{9pt plus 2pt minus 3pt}
  \setlength{\belowdisplayskip}{9pt plus 2pt minus 3pt}
  \setlength{\abovedisplayshortskip}{6pt plus 2pt minus 2pt}
  \setlength{\belowdisplayshortskip}{6pt plus 2pt minus 2pt}
  \setlength{\jot}{4pt}
}

\title[Ranked spreadness and sample-based testing]
{Ranked spreadness and sample-based testing}
\author[G. Carenini]{Gaia Carenini}
\address{Trinity College Cambridge, Department of Pure Mathematics and Mathematical Statistics,
Centre for Mathematical Sciences, Wilberforce Road, Cambridge CB3 0WA, United Kingdom}
\email{gc645@cam.ac.uk}
\date{\today}

\begin{document}

\begin{abstract}
In this note, we introduce the notion of \emph{ranked spreadness}, a strengthening of the usual spread
condition in which the elements of each member can be ordered so that their
one-coordinate marginals decay geometrically with their rank.  This additional
structure removes the dependence on the maximum set size in random-containment
estimates.  We prove width-free hitting and weighted-concentration theorems for
ranked-spread set systems, together with an elementary kernel-extraction theorem
showing that ranked spreadness arises naturally in arbitrary distributions on
small sets.

Our main application is to the simulation of nonadaptive property testers
by sample-based testers.  If a one-sided tester has average query complexity
$d$ and rejects every far input with probability at least $\delta$, then, for
every integer $c>d/\delta$, it admits a one-sided sample-based simulation with
expected sample complexity
$O_{d,\delta,|\Sigma|}\bigl(n^{1-1/c}\bigr)$.
More generally, if positive inputs are rejected with probability at most
$\gamma$ and far inputs with probability at least $\delta>\gamma$, the same
conclusion holds for every $c>d/(\delta-\gamma)$.  In particular, for
constant-query nonadaptive testers we obtain an exponent $1-\Theta(1/q)$,
matching, up to the dependence on the rejection gap, the exponent conjectured
by Fischer, Lachish, and Vasudev.
\end{abstract}

\maketitle

\section{Introduction}

Spreadness is a basic pseudorandomness condition for set systems.  In one
standard normalization, a weighted family $\mathcal F$ is $(m,k)$-spread if
$$
  \mu\bigl(\{F\in\mathcal F:A\subseteq F\}\bigr)
  \le m k^{-|A|}
$$
for every nonempty set $A$.  Such bounds control the clustering of members of
$\mathcal F$ and underlie the robust-sunflower method, expectation-threshold
results, and a number of sampling and local-to-global arguments
\cite{alweiss2021improved,rao2020coding,frankston2021thresholds,
parkpham2024proof,mossel2025bayesian}.

A persistent feature of general spread-to-hitting theorems is a dependence on
the maximum size of a member.  If every member has size at most $r$, ordinary
spreadness typically guarantees random containment only at density of order
$(\log r)/k$.  This logarithmic loss is necessary for unrestricted spread
families, but it can be artificial when the spread condition comes from an
underlying geometric hierarchy of one-coordinate marginals.

In this work, we isolate such a hierarchy through the notion of
\emph{ranked spreadness}.  A weighted indexed family $(P_t)_{t\in\calT}$ is
ranked-spread if, for every $t$, the elements of $P_t$ can be ordered as
$P_t=\{x_1(t),\ldots,x_{|P_t|}(t)\}
$
so that the marginal of the element in rank $i$ is at most $mk^{-i}$.  The
ordering is allowed to depend on the member $P_t$; there need not be a global
rank assigned to the ground-set elements.  Ranked spreadness implies ordinary
spreadness with the same parameters, but retains considerably more information.

The reason why this extra information is useful is fairly simple.  When two members
intersect, we charge their interaction to the last-ranked element of the first
member lying in the second.  At sampling density $p=\alpha/k$, this bounds the
Janson dependency sum by a multiple of
$$
  \sum_{i=1}^{r}(pk)^{-i}
  =
  \sum_{i=1}^{r}\alpha^{-i}.
$$
For every $\alpha>1$, this is bounded independently of $r$.  Thus ranked
spreadness converts the width-dependent loss in the usual spread lemma into a
convergent geometric series.

\subsection{The ranked-spread method}

We develop three complementary forms of this principle. The first is a \emph{random-hitting theorem}.  If an indexed family is
$(m,k)$-ranked-spread and $\calA$ is a subfamily of mass $\tau$, then a
Bernoulli sample of density $p=\alpha/k$ contains some member of $\calA$ with
failure probability at most
$\exp\left(
    -{\tau}/{2m\sum_{i=1}^{r}\alpha^{-i}}
  \right)$.
In particular, whenever $\alpha>1$, the failure probability is at most
$\exp\left(-{(\alpha-1)\tau}/{2m}\right)$,
with no dependence on the maximum member size $r$.

The second is a \emph{weighted-concentration theorem}.  For a sampled set $W$, define
$$
  X_{\calA}
  =
  \sum_{t\in\calA}
  \mu(t)p^{-|P_t|}
  \one_{\{P_t\subseteq W\}}.
$$
This is an unbiased estimator for the mass $\mu(\calA)$.  Ranked spreadness
controls both its variance and its lower tail, again through the geometric
series $\sum_i\alpha^{-i}$.  This weighted statement is strictly more
informative than the assertion that at least one member is exposed, since it permits
one to estimate the total mass of an indexed subfamily from a random sample.

The third ingredient is a \emph{kernel-extraction theorem}.  Given an arbitrary
distribution on sets of average size $d$, and an integer $c>d$, we find a
scale $k=n^{1/c}$, a kernel $K$ of size at most $n^{1-1/c}$, and a subfamily of
mass at least $1-d/c$ such that the petals obtained after removing $K$ are
ranked-spread.  The proof is elementary.  We divide the one-coordinate
marginals into geometric scales and use an interval-packing argument to choose
a boundary at which most members exhibit the required rank-by-rank decay.

\subsection{Robust daisies}

Our first application is a strengthened robust daisy lemma.  Robust daisies
were introduced by Goldberg, Gur, and Saraogi
\cite{goldberg2026robust} in order to retain random-hitting guarantees
simultaneously for every subfamily of a distribution on small sets.  Their
framework already yields one-sided sample-based simulations, but its general
extraction theorem incurs losses depending on the set width and on the size of
the support.

Ranked kernel extraction preserves the more precise marginal hierarchy hidden
in the extraction procedure.  Combining it with the ranked hitting theorem
gives robustness at every density
$p=\alpha n^{-1/c}$ with $\alpha>1$,
and produces a failure exponent independent of both the support size and the
maximum petal width.  In the normalization of
\cite[Lemma~5.1]{goldberg2026robust}, their stated bound applies for
$\alpha>2q$ and has exponent proportional to
${\alpha}/{(q^2\log|\calD|)}\cdot|K|$.
Our bound applies for every $\alpha>1$ and has exponent proportional to
${(\alpha-1)}/{q}\cdot B$,
where $B$ is the extraction scale and may be larger than the kernel itself.
Thus the ranked-spread formulation both simplifies the extraction and
identifies the parameter naturally governing robustness.

\subsection{Application to sample-based testing}

Our principal application concerns the conversion of nonadaptive property
testers into sample-based testers.  Property testing, introduced by Goldreich,
Goldwasser, and Ron~\cite{goldreichgoldwasserron1998}, asks how efficiently one
can distinguish inputs satisfying a property from inputs far from it.  A
standard tester chooses a small set of coordinates to query.  A
\emph{sample-based tester} instead observes the values on a uniform Bernoulli
sample of the coordinates and then performs arbitrary post-processing.  This
model was systematically studied by Goldreich and
Ron~\cite{goldreichron2016sample}, who asked when arbitrary testers can be
simulated from random samples.

Fischer, Lachish, and Vasudev~\cite{fischerlachishvasudev2015} proved that
every constant-query nonadaptive tester over a fixed alphabet has a sublinear
sample-based simulation.  In the one-sided case, their direct argument uses
sampling probability of order $n^{-1/q^2}$, and they conjectured that the
correct exponent should be $1/q$.  Dall'Agnol, Gur, and
Lachish~\cite{dallagnol2023structural} subsequently treated adaptive robust
local algorithms, obtaining exponent $1/O(q^2\log^2 q)$ in that more general
setting.  The robust-daisy machinery of Goldberg, Gur, and Saraogi also implies
a one-sided conversion, with slightly weaker quantitative parameters than those
obtained here.

We prove a linear, rather than quadratic, dependence on the query complexity.
More precisely, our bounds depend on the \emph{average} number of queries, as
well as on the rejection gap.

Let $\Pi\subseteq\Sigma^n$ be a property over a finite alphabet, and let
$r_T(x)$ denote the rejection probability of a nonadaptive tester $T$ on input
$x$.  We first state the one-sided result.

\begin{theorem}[One-sided query-to-sample conversion]
\label{thm:one-sided-main}
Let $\Sigma$ be a finite alphabet, let $\eps,\delta\in(0,1]$, and let $T$ be
a one-sided nonadaptive $\eps/2$-tester for $\Pi\subseteq\Sigma^n$.  Suppose
that $T$ has average query complexity $d>0$ and rejects every
$\eps/2$-far input with probability at least $\delta$.  Let $c$ be an integer
with $c>d/\delta$, and let
$\beta=\delta-\frac dc>0$, and 
$\alpha_{\mathrm{hit}}
  =1+\frac{2d}{\beta}\bigl(\log|\Sigma|+\log3\bigr)$.
Then, for all sufficiently large $n$, $\Pi$ has a one-sided $\eps$-tester
which samples every coordinate independently with probability
$p=\alpha_{\mathrm{hit}}n^{-1/c}$.  It rejects every $\eps$-far input with
probability at least $2/3$ and has expected sample complexity
$\alpha_{\mathrm{hit}}n^{1-1/c}$.
\end{theorem}

The one-sided simulation uses only the hitting form of ranked spreadness.
After extracting a small kernel, the new tester considers every possible
assignment to the kernel.  For each false assignment, the rejecting tests form
a subfamily of positive mass, and the ranked hitting theorem ensures that the
sample exposes one of their petals.  Perfect completeness guarantees that the
true kernel assignment is never eliminated.

The two-sided setting requires more.  A valid input may itself expose
rejecting tests, so the existence of a single rejecting local view carries no
useful information.  Instead, the tester must estimate the mass of rejecting
tests associated with each possible kernel assignment.  This is precisely the
role of weighted ranked concentration.

\begin{theorem}[Two-sided query-to-sample conversion]
\label{thm:two-sided-main}
Let $\Sigma$ be a finite alphabet, let $\eps\in(0,1]$, and let
$0\le\gamma<\delta\le1$.  Let $T$ be a nonadaptive tester for
$\Pi\subseteq\Sigma^n$.  Suppose that $T$ has average query complexity
$d>0$, rejects every input in $\Pi$ with probability at most $\gamma$, and
rejects every $\eps/2$-far input with probability at least $\delta$.  Let
$c$ be an integer with
$c>\frac{d}{\delta-\gamma}$,
and let $\beta=\delta-\frac dc$,
$g=\beta-\gamma>0$,
and
$\alpha_{\mathrm{wt}}=1+\frac{4d}{g^2}\bigl(3\gamma+2\log(3|\Sigma|)\bigr)$.
Then, for all sufficiently large $n$, $\Pi$ has a two-sided $\eps$-tester
which samples every coordinate independently with probability
$p=\alpha_{\mathrm{wt}}n^{-1/c}$.
It accepts every input in $\Pi$ with probability at least $2/3$, rejects
every $\eps$-far input with probability at least $2/3$, and has expected
sample complexity
$\alpha_{\mathrm{wt}}n^{1-1/c}$.
\end{theorem}

Taking $d\le q$ gives the following convenient worst-case formulation.

\begin{corollary}
\label{cor:q-formss}
A one-sided $q$-query tester with detection probability $\delta$ admits a
one-sided sample-based simulation with expected sample complexity
$O_{q,\delta,|\Sigma|}
  \left(n^{1-1/(\lfloor q/\delta\rfloor+1)}\right)$.
More generally, a two-sided $q$-query tester with positive-input rejection
probability at most $\gamma$ and far-input rejection probability at least
$\delta$ admits a simulation with expected sample complexity
$O_{q,\gamma,\delta,|\Sigma|}
  \left(n^{1-1/(\lfloor q/(\delta-\gamma)\rfloor+1)}\right)$. 
For $\gamma=1/3$ and $\delta=2/3$, this is
$O_{q,|\Sigma|}\bigl(n^{1-1/(3q+1)}\bigr)
  =
  n^{1-\Theta(1/q)}.
$
\end{corollary}

The distinction between the two conversions is important.  When $\gamma=0$,
the weighted theorem recovers the same exponent as the one-sided theorem, but
with a sampling constant proportional to the inverse square of the retained
rejection margin.  The one-sided argument uses perfect completeness and only
needs to expose one witness for each false kernel assignment.  The two-sided
argument has no such Boolean elimination rule and instead relies on the
importance-weighted estimator supplied by ranked concentration.  In this
sense, the two testing results illustrate the two basic probabilistic
consequences of ranked spreadness: hitting and mass estimation.

\subsection{Organization}

In \cref{sec:preliminaries}, we introduce the testing model and recall the
unweighted and weighted forms of Janson's inequality used throughout.  In
\cref{sec:ranked}, we define ranked spreadness and prove the width-free hitting
and weighted-concentration theorems.  In \cref{sec:extraction}, we prove ranked
kernel extraction, derive its hitting and concentration consequences, and
obtain the improved robust daisy lemma.  In \cref{sec:simulation}, we apply the
framework to one-sided and two-sided query-to-sample conversion.  

\subsection*{Acknowledgments}
The author is grateful to Guy Goldberg, Tom Gur, and Sidhant Saraogi for helpful discussions that clarified the relationship between the present work and theirs. She is especially grateful to Tom Gur for encouraging her to investigate applications of an earlier version of the main result, which led to a substantial strengthening of the paper. This work was supported by the CB European PhD Studentship funded by Trinity College, Cambridge.
\section{Preliminaries}
\label{sec:preliminaries}

Let $\Sigma$ be a finite alphabet.  For $x,y\in\Sigma^{n}$, their normalized
Hamming distance is
$$
  \dist(x,y)=\frac1n|\{i\in[n]:x_i\ne y_i\}|.
$$
For a property $\Pi\subseteq\Sigma^{n}$, write
$$
  \dist(x,\Pi)=\min_{y\in\Pi}\dist(x,y).
$$
The word $x$ is $\eps$-far from $\Pi$ if $\dist(x,\Pi)>\eps$.

A deterministic \emph{local test} is a pair
$t=(Q_t,\varphi_t),
$
where $Q_t\subseteq[n]$ and
$\varphi_t:\Sigma^{Q_t}\longrightarrow\{0,1\}$.
We interpret $\varphi_t=1$ as rejection.  A finite probability distribution
$\mu$ on a set $\calT$ of local tests defines a nonadaptive randomized tester:
it samples $t\sim\mu$, queries $Q_t$, and outputs $\varphi_t(x|_{Q_t})$.
Its rejection probability on $x$ is
$r_T(x)=
  \mu\bigl(\{t\in\calT:\varphi_t(x|_{Q_t})=1\}\bigr)$. The maximum and average query complexities are, respectively,
$q=\max_{t\in\supp(\mu)}|Q_t|$ and $d=\E_{t\sim\mu}|Q_t|$.

Every finite-randomness nonadaptive tester has this form after including all
random coins in the index $t$.  More generally, since there are only finitely
many pairs consisting of a query set of size at most $q$ and a deterministic
predicate on that set, arbitrary internal randomness can be grouped into a
finite distribution of local tests.

The tester is \emph{one-sided} for $\Pi$ if
$r_T(x)=0$ for every $x\in\Pi$.
Equivalently, every rejecting local view of every positive-probability local
test is inconsistent with membership in $\Pi$. For $p\in[0,1]$, let $W\sim\Bin([n],p)$ denote the random set obtained by
retaining each coordinate independently with probability $p$.  A
\emph{$p$-sampling tester} receives $(W,x|_W)$ and applies arbitrary
post-processing.  Its expected sample complexity is $pn$.

This is the sample-based model of Goldreich and Ron
\cite{goldreichron2016sample}.  We do not impose a running-time bound on the
post-processing.

The combinatorial arguments use the query sets but must retain the indexing. In particular, two local tests may use the same query set and different predicates.  Thus we
work with a probability space $(\calT,\mu)$ and a set $P_t\subseteq U$
attached to every index $t\in\calT$.  For $x\in U$, write
$\mu(x)=\mu\bigl(\{t\in\calT:x\in P_t\}\bigr)$.
For $\calA\subseteq\calT$, we use $\mu(\calA)$ for its probability mass.

We use the following standard form of Janson's inequality; see, for example,
\cite[Chapter~8]{alonspencer2016}.

\begin{lemma}[Janson]
\label{lem:janson}
Let $\calM$ be a finite family of subsets of a finite ground set, and let
$W$ contain each ground-set element independently with probability $p$.
For $M\in\calM$, let $Z_M=\one_{\{M\subseteq W\}}$, let
$Z=\sum_{M\in\calM}Z_M$, $ \lambda=\E Z$, and define the ordered dependency sum
$$
  \Delta=
  \sum_{\substack{M,N\in\calM:\ M\ne N,\ M\cap N\ne\emptyset}}
  \E[Z_MZ_N].
$$
Then
$$
  \Prb[Z=0]
  \le
  \exp\left(
    -\min\left\{
      \frac\lambda2,
      \frac{\lambda^2}{2\Delta}
    \right\}
  \right),
$$
with the second term interpreted as $+\infty$ when $\Delta=0$.
\end{lemma}

\begin{lemma}[Weighted lower-tail Janson]
\label{lem:weighted-janson}
Let $(A_i)_{i\in I}$ be a finite indexed family of nonempty subsets of a finite
ground set, let $W$ include every ground-set element independently with
probability $p$, and let $w_i\ge0$.  Let
$I_i=\one_{\{A_i\subseteq W\}}$, $X=\sum_{i\in I}w_iI_i$, $\lambda=\E X,
$
and define the ordered dependency sum, including the diagonal, by
$$
  \overline\Delta
  =
  \sum_{\substack{i,j\in I\\A_i\cap A_j\ne\emptyset}}
  w_iw_j\E[I_iI_j].
$$
Then, for every $\eta\in[0,1]$,
\begin{equation}
\label{eq:weighted-janson}
  \Prb[X\le(1-\eta)\lambda]
  \le
  \exp\left(
    -\frac{\eta^2\lambda^2}{2\overline\Delta}
  \right),
\end{equation}
with the natural interpretation when $\overline\Delta=0$.
\end{lemma}

\begin{proof}
Suppose first that every $w_i$ is rational.  Choose $N$ so that $Nw_i$ is an
integer for all $i$, and replace $A_i$ by $Nw_i$ indexed copies.  If $Y$ is
the number of sampled copies, then $Y=NX$, and $\E Y=N\lambda$.
The ordered dependency sum for the copied family, including the diagonal, is
$N^2\overline\Delta$.  The standard lower-tail Janson inequality therefore
gives
$$
  \Prb[X\le(1-\eta)\lambda]
  =
  \Prb[Y\le(1-\eta)\E Y]
  \le
  \exp\left(
    -\frac{\eta^2N^2\lambda^2}{2N^2\overline\Delta}
  \right).
$$
For arbitrary nonnegative weights, approximate the weights by nonnegative
rationals.  Since the index set is finite, the random variables, their means,
and their dependency sums converge uniformly; passing to the limit, with an
arbitrarily small slack in the threshold, gives
\eqref{eq:weighted-janson}.
\end{proof}

\section{Ranked spreadness and random hitting}
\label{sec:ranked}
We start by introducing the notion of \emph{ranked spreadness}.
\begin{definition}[Ranked spreadness]
\label{def:ranked-spread}
Let $m,k>0$.  A weighted indexed set system
$(P_t)_{t\in\calT}$ with distribution $\mu$ is
\emph{$(m,k)$-ranked-spread} if every nonempty $P_t$ admits an ordering
$P_t=\{x_1(t),\ldots,x_{|P_t|}(t)\}
$
such that, for every $i$,
$\mu\bigl(\{s\in\calT:x_i(t)\in P_s\}\bigr)
  \le
  m k^{-i}.$
\end{definition}

In the normalization used here, a weighted indexed set system is
\emph{$(m,k)$-spread} if
$\mu\bigl(\{s\in\calT:A\subseteq P_s\}\bigr)
  \le m k^{-|A|}$
for every nonempty $A\subseteq U$.  This is the usual spread condition from
the robust-sunflower and expectation-threshold literature
\cite{alweiss2021improved,rao2020coding,frankston2021thresholds}.

Ranked spreadness implies ordinary spreadness with the same parameters.  To
see this, suppose that $A\subseteq P_t$, and let $x_i(t)$ be the element of
$A$ having largest rank in the ranked-spread ordering of $P_t$.  Since $A$
contains $|A|$ elements, $i\ge |A|$, and hence
\begin{align*}
  \mu\bigl(\{s:A\subseteq P_s\}\bigr)
  \le \mu\bigl(\{s:x_i(t)\in P_s\}\bigr)\le m k^{-i}
  \le m k^{-|A|}.
\end{align*}
Thus ranked spreadness is a genuine strengthening of ordinary spreadness.

For an arbitrary spread system whose members have size at most $r$, the
general spread lemma gives random containment at density of order
$(\log r)/k$; equivalently, the general expectation-threshold theory retains
a logarithmic dependence on the largest minimal witness
\cite{frankston2021thresholds,parkpham2024proof,mossel2025bayesian}.  The
additional edgewise hierarchy is what allows \cref{thm:ranked-hitting} to work
at every density $\alpha/k$ with $\alpha>1$, with an exponential failure
bound independent of $r$.  This hierarchy is distinct from other multiscale
refinements of spreadness, such as Spiro's smoother spread condition
\cite{spiro2023smoother}, which controls codegrees across prescribed ranges of
intersection sizes.

The ordering may depend on $t$; there need not be a globally consistent rank
assigned to each coordinate. For $r\ge1$ and $\alpha>0$, let
$H_r(\alpha)=\sum_{i=1}^{r}\alpha^{-i}$.

\begin{theorem}[Ranked spread-to-hitting]
\label{thm:ranked-hitting}
Let $(P_t)_{t\in\calT}$ be an $(m,k)$-ranked-spread indexed set system, and
suppose that $|P_t|\le r$ for every $t$.  Let
$\calA\subseteq\calT$, put $\tau=\mu(\calA)$, and let
$W\sim\Bin(U,p)$, and $p=\frac\alpha k\le1$. Then
\begin{equation}
\label{eq:ranked-general}
  \Prb_W\bigl[\nexists t\in\calA:\ P_t\subseteq W\bigr]
  \le
  \exp\left(
    -\frac{\tau}{2mH_r(\alpha)}
  \right).
\end{equation}
If $\alpha>1$, then
\begin{equation}
\label{eq:ranked-supercritical}
  \Prb_W\bigl[\nexists t\in\calA:\ P_t\subseteq W\bigr]
  \le
  \exp\left(
    -\frac{(\alpha-1)\tau}{2m}
  \right).
\end{equation}
\end{theorem}

\begin{proof}
If $\tau=0$, there is nothing to prove.  If $P_t=\emptyset$ for some
$t\in\calA$, the event on the left has probability zero.  We may therefore
assume that every set indexed by $\calA$ is nonempty.

Condition on $\calA$.  The resulting distribution $\mu_{\calA}$ is
$(m/\tau,k)$-ranked-spread, using the same edgewise orderings.  It is therefore
enough to prove the theorem when $\calA=\calT$ and then replace $m$ by
$m/\tau$.

We first suppose that all values $\mu(t)$ are rational.  Choose a multiset
$\calM_0$ in which the proportion of copies of each index $t$ is $\mu(t)$.
Let
$r_0=\max_{t\in\calT}|P_t|\le r$.

For an integer $L\ge1$, replace every member of $\calM_0$ by $L$ copies.  Pad
each copy of $P_t$ to size $r_0$ by adding $r_0-|P_t|$ fresh dummy elements,
used only by that copy.  Let $\calM_L$ be the resulting uniform multiset on
the enlarged ground set.  The dummy elements are mutually distinct and do
not belong to $U$.

Sample every element of the enlarged ground set independently with
probability $p$.  If a padded set is sampled, then its original set is
contained in $W\cap U$.  Hence
\begin{equation}
\label{eq:padding-domination}
  \Prb_W[\nexists t:\ P_t\subseteq W]
  \le
  \Prb[\text{no member of }\calM_L\text{ is sampled}].
\end{equation}

Let $N_L=|\calM_L|$ and, for $M\in\calM_L$, let $Z_M$ be the indicator that
$M$ is sampled.  Then
$$
  \lambda=\E\sum_{M\in\calM_L}Z_M=N_Lp^{r_0}.
$$
Fix $M\in\calM_L$, and order its original elements according to the
ranked-spread ordering of the corresponding $P_t$:
$x_1(M),\ldots,x_{|P_t|}(M)$.
If another padded member $N$ intersects $M$, their intersection contains no
dummy element.  Let
$i(M,N)=
  \max\{i:x_i(M)\in N\}$.
Then $|M\cap N|\le i(M,N)$.  Moreover, the number of $N\in\calM_L$ with
$i(M,N)=i$ is at most the number containing $x_i(M)$, and hence at most
$N_L m k^{-i}$.
Consequently, the ordered Janson dependency sum satisfies
\begin{align}
  \Delta
  \le
  \sum_{M\in\calM_L}
  \sum_{i=1}^{r_0}
  N_L m k^{-i}p^{2r_0-i}
=
  N_L^2p^{2r_0}m
  \sum_{i=1}^{r_0}(pk)^{-i}
  =
  \lambda^2mH_{r_0}(\alpha)
  \le
  \lambda^2mH_r(\alpha).
\label{eq:ranked-delta}
\end{align}
By \cref{lem:janson},
$$
  \Prb[Z=0]
  \le
  \exp\left(
    -\min\left\{
      \frac{N_Lp^{r_0}}2,
      \frac{1}{2mH_r(\alpha)}
    \right\}
  \right).
$$
The first term tends to infinity as $L\to\infty$.  Combining this with
\eqref{eq:padding-domination} gives
$$
  \Prb_W[\nexists t:\ P_t\subseteq W]
  \le
  \exp\left(-\frac{1}{2mH_r(\alpha)}\right).
$$

For a general finite distribution, approximate its atom weights by positive
rational weights with the same support.  The coordinate marginals converge,
so the approximating systems are $(m+o(1),k)$-ranked-spread.  The event to be
bounded depends only on the indexed support and the sets $P_t$, not on the
weights.  Letting the approximation error tend to zero proves
\eqref{eq:ranked-general}.

Finally, if $\alpha>1$, then
$$
  H_r(\alpha)
  =\sum_{i=1}^{r}\alpha^{-i}
  \le\frac1{\alpha-1},
$$
which gives \eqref{eq:ranked-supercritical}.
\end{proof}

Observe that, at the critical density $p=1/k$, the theorem gives an exponent of order
$1/(mr)$.  Further, for $p=\alpha/k$ with $\alpha>1$, the dependence on $r$ disappears.

\subsection{Weighted ranked concentration}
\label{subsec:weighted-ranked}

The hitting theorem only records whether some petal is exposed.  For two-sided testing we need to estimate the total mass of a rejecting subfamily.  The following importance-weighted version provides exactly this lower-tail control.

\begin{theorem}[Weighted ranked concentration]
\label{thm:weighted-ranked}
Let $(P_t)_{t\in\calT}$ be an $(m,k)$-ranked-spread indexed set system with
distribution $\mu$, and suppose that $|P_t|\le r$ for every $t$.  Let
$\calA\subseteq\calT$, put $\tau=\mu(\calA)$, and let
$W\sim\Bin(U,p)$, with $p=\frac{\alpha}{k}\le1.
$
Define $X_{\calA}
  =
  \sum_{t\in\calA}
  \mu(t)p^{-|P_t|}\one_{\{P_t\subseteq W\}}$.

Then, $\E X_{\calA}=\tau$, $\operatorname{Var}(X_{\calA})
  \le
  \tau mH_r(\alpha)$, and, for every $\eta\in[0,1]$,
\begin{equation}
\label{eq:ranked-lower-tail}
  \Prb\bigl[X_{\calA}\le(1-\eta)\tau\bigr]
  \le
  \exp\left(
    -\frac{\eta^2\tau}{2mH_r(\alpha)}
  \right).
\end{equation}
\end{theorem}

\begin{proof}
Indices with $P_t=\emptyset$ contribute deterministically to
$X_{\calA}$.  Let their total mass be $\tau_0$, and put
$\tau_+=\tau-\tau_0$.  If $(1-\eta)\tau<\tau_0$, then the event in
\eqref{eq:ranked-lower-tail} is empty.  Otherwise it is a lower-tail event for
the positive-petal contribution with relative deviation
$\eta_+=\frac{\eta\tau}{\tau_+}\le1$.
Since
$\eta_+^2\tau_+
  ={\eta^2\tau^2}/{\tau_+}
  \ge\eta^2\tau$,
it is enough to treat the case in which every $P_t$, $t\in\calA$, is
nonempty.

For $t\in\calA$, set
$I_t=\one_{\{P_t\subseteq W\}}$, and $w_t=\mu(t)p^{-|P_t|}$. Then $X_{\calA}=\sum_{t\in\calA}w_tI_t$, and
$$
  \E[w_tI_t]
  =\mu(t)p^{-|P_t|}p^{|P_t|}
  =\mu(t),
$$
which proves that $\E X_{\calA}=\tau$.

Define
\begin{align}
  \Lambda_{\calA}
  &:={}
  \sum_{\substack{t,s\in\calA\\P_t\cap P_s\ne\emptyset}}
  w_tw_s\E[I_tI_s]=
  \sum_{\substack{t,s\in\calA\\P_t\cap P_s\ne\emptyset}}
  \mu(t)\mu(s)p^{-|P_t\cap P_s|}.
\label{eq:lambda-dependency}
\end{align}
Fix $t\in\calA$, and use a ranked-spread ordering
$P_t=\{x_1(t),\ldots,x_{|P_t|}(t)\}.
$
For every $s$ with $P_s\cap P_t\ne\emptyset$, let
$i(t,s)=\max\{i:x_i(t)\in P_s\}$.
Then $|P_t\cap P_s|\le i(t,s)$.  Since $p\le1$,
$p^{-|P_t\cap P_s|}\le p^{-i(t,s)}$. Moreover,
$$
  \sum_{\substack{s\in\calA\\i(t,s)=i}}\mu(s)
  \le
  \mu\bigl(\{s:x_i(t)\in P_s\}\bigr)
  \le mk^{-i}.
$$
Consequently,
\begin{align}
  \sum_{\substack{s\in\calA\\P_t\cap P_s\ne\emptyset}}
  \mu(s)p^{-|P_t\cap P_s|}
  \le
  m\sum_{i=1}^{|P_t|}k^{-i}p^{-i}
  \le mH_r(\alpha).
\label{eq:one-sided-dependency}
\end{align}
Multiplying by $\mu(t)$ and summing over $t\in\calA$ gives
\begin{equation}
\label{eq:dependency-final}
  \Lambda_{\calA}
  \le
  \tau mH_r(\alpha).
\end{equation}

If $P_t\cap P_s=\emptyset$, then $I_t$ and $I_s$ are independent.  For an
intersecting pair, $\operatorname{Cov}(I_t,I_s)\le\E[I_tI_s]$. It follows from \eqref{eq:dependency-final} that
$\operatorname{Var}(X_{\calA})
  \le\Lambda_{\calA}
  \le\tau mH_r(\alpha),
$
proving $ \operatorname{Var}(X_{\calA})
  \le
  \tau mH_r(\alpha)$.

Applying \cref{lem:weighted-janson} with the weights $w_t$ and using
\eqref{eq:dependency-final}, we obtain
$$
  \Prb[X_{\calA}\le(1-\eta)\tau]
  \le
  \exp\left(
    -\frac{\eta^2\tau^2}{2\Lambda_{\calA}}
  \right)
  \le
  \exp\left(
    -\frac{\eta^2\tau}{2mH_r(\alpha)}
  \right).
$$

\end{proof}

\section{Ranked kernel extraction}
\label{sec:extraction}

We now show that every distribution on small indexed sets contains a large
ranked-spread subfamily after the removal of a small kernel.  The boundary
selection is inspired by the robust-daisy extraction of Goldberg, Gur, and
Saraogi~\cite{goldberg2026robust}.  The new point is to retain the ordered
one-coordinate marginal information produced by the proof.

\begin{lemma}[Interval packing]
\label{lem:interval-packing}
Let $A$ be a finite multiset of integers.  Call an integer $j$ bad if there is
$r\ge1$ such that
$|A\cap\{j-r,j-r+1,\ldots,j-1\}|\ge r,
$
where multiplicities are counted.  Then the number of bad integers is at most
$|A|$.
\end{lemma}

\begin{proof}
Greedily process the bad integers from right to left.  Let $j$ be the largest
remaining bad integer, and choose $r\ge1$ witnessing its badness.  Remove all
remaining bad integers in
$\{j-r+1,\ldots,j\}$.
At most $r$ bad integers are removed, while the interval
$\{j-r,\ldots,j-1\}$ contains at least $r$ elements of $A$, counted with
multiplicity.  Charge the removed bad integers to distinct such elements.

At the next step, the chosen bad integer is at most $j-r$, so its charging
interval lies strictly to the left of the preceding one.  Thus the charging
intervals are pairwise disjoint.  The total number of removed bad integers is
therefore at most $|A|$.
\end{proof}

\begin{theorem}[Ranked kernel extraction]
\label{thm:ranked-extraction}
Let $(Q_t)_{t\in\calT}$ be an indexed set system on a ground set $U$ of size
$n$, equipped with a probability distribution $\mu$.  Suppose that
$|Q_t|\le q$ for every $t$, and let
$d=\E_{t\sim\mu}|Q_t|$. Assume $d>0$, let $c>d$ be an integer, and set
$k=n^{1/c}$.
Then there exist $j\in\{1,\ldots,c\}$, a subfamily
$D\subseteq\calT$, a kernel $K\subseteq U$, and a scale
$B={n}/{k^j}=n^{1-j/c}$ such that
$\mu(D)\ge1-d/c$, and $|K|\le B\le n^{1-1/c}$. Moreover, the indexed petal system
$P_t=Q_t\setminus K$, for all $t\in D$,
under the conditioned distribution $\mu_D$, is
$\left({d}/{\mu(D)B},k\right)
$
-ranked-spread.
\end{theorem}

\begin{proof}
For $x\in U$, define its one-coordinate marginal
$\delta(x)=\mu\bigl(\{t:x\in Q_t\}\bigr)$.
Then
\begin{equation}
\label{eq:sum-marginals}
  \sum_{x\in U}\delta(x)
  =\E_{t\sim\mu}|Q_t|
  =d.
\end{equation}
For every integer $s$, let
$\lambda_s=d\frac{k^s}{n}$.
For $j\in\{1,\ldots,c\}$, define
$K_j=\{x\in U:\delta(x)\ge\lambda_j\}$.
By \eqref{eq:sum-marginals},
$|K_j|\lambda_j\le d$,
and hence
\begin{equation}
\label{eq:kernel-size}
  |K_j|\le\frac{d}{\lambda_j}
  =\frac{n}{k^j}.
\end{equation}

We say that $j$ is \emph{good} for an index $t$ if every nonempty
$R\subseteq Q_t\setminus K_j$ contains some $x\in R$ satisfying
\begin{equation}
\label{eq:good-boundary}
  \delta(x)<\lambda_{j-|R|}.
\end{equation}
Otherwise, $j$ is bad for $t$.

Fix $t\in\calT$.  Every $x\in Q_t$ has positive marginal, because $t$ has
positive probability.  Choose an integer $\ell(x)$ such that
$\lambda_{\ell(x)}
  \le\delta(x)<\lambda_{\ell(x)+1}$,
and let $A_t$ be the multiset of the levels $\ell(x)$, for $x\in Q_t$.
If $j$ is bad for $t$, there is a nonempty
$R\subseteq Q_t\setminus K_j$, say $|R|=r$, such that
$\delta(x)\ge\lambda_{j-r}$ for every $x\in R$.
Since $R\cap K_j=\emptyset$, we also have $\delta(x)<\lambda_j$.  Therefore
$j-r\le\ell(x)\le j-1$ for every $x\in R$.
The interval $\{j-r,\ldots,j-1\}$ thus contains at least $r$ elements of
$A_t$.  By \cref{lem:interval-packing}, the number of integers $j$ that are
bad for $t$ is at most
$|A_t|=|Q_t|$.

Let
$D_j=\{t\in\calT:j\text{ is good for }t\}$.
Averaging over $j\in[c]$ and then over $t\sim\mu$ gives
\begin{align*}
  \frac1c\sum_{j=1}^{c}\mu(D_j)
  &=1-\frac1c
    \E_{t\sim\mu}
    |\{j\in[c]:j\text{ is bad for }t\}|\\
  &\ge1-\frac1c\E_{t\sim\mu}|Q_t|\\
  &=1-\frac dc.
\end{align*}
Choose $j$ with $\mu(D_j)\ge1-d/c$, and let
$D=D_j$, $K=K_j$, and 
$B=n/{k^j}$.
The asserted kernel bounds follow from \eqref{eq:kernel-size}.

It remains to verify ranked spreadness.  Fix $t\in D$, and order the elements
of its petal $P_t=Q_t\setminus K$ by nonincreasing marginal:
$\delta(x_1(t))\ge\delta(x_2(t))\ge\cdots
  \ge\delta(x_{|P_t|}(t))$.

For $i\in[|P_t|]$, apply \eqref{eq:good-boundary} to
$R_i=\{x_1(t),\ldots,x_i(t)\}$.
Some $x\in R_i$ satisfies $\delta(x)<\lambda_{j-i}$.  Since $x_i(t)$ has the
smallest marginal in $R_i$,
$$
  \delta(x_i(t))<\lambda_{j-i}
  =d\frac{k^{j-i}}n
  =\frac dB k^{-i}.
$$
As $x_i(t)\notin K$, membership in a petal is the same as membership in the
original query set.  Therefore
\begin{align*}
  \mu_D\bigl(\{s\in D:x_i(t)\in P_s\}\bigr)
  \le\frac{\delta(x_i(t))}{\mu(D)}<\frac{d}{\mu(D)B}k^{-i}.
\end{align*}
This is precisely the claimed ranked-spread property.
\end{proof}

Combining extraction with \cref{thm:ranked-hitting} gives the form used in
the testing argument.

\begin{corollary}[Ranked kernel hitting]
\label{cor:ranked-kernel-hitting}
Under the hypotheses and notation of \cref{thm:ranked-extraction}, let
$\calA\subseteq D$ and let
$W\sim\Bin\left(U,\frac\alpha k\right)$,
$\alpha/k\le1$.
If $r=\max_{t\in D}|Q_t\setminus K|$, then
\begin{equation}
\label{eq:kernel-hitting-exact}
  \Prb_W\bigl[
    \nexists t\in\calA:\ Q_t\setminus K\subseteq W
  \bigr]
  \le
  \exp\left(
    -\frac{B\mu(\calA)}{2dH_r(\alpha)}
  \right).
\end{equation}
In particular, if $\alpha>1$, then
\begin{equation}
\label{eq:kernel-hitting-simple}
  \Prb_W\bigl[
    \nexists t\in\calA:\ Q_t\setminus K\subseteq W
  \bigr]
  \le
  \exp\left(
    -\frac{(\alpha-1)B}{2d}\mu(\calA)
  \right).
\end{equation}
\end{corollary}

\begin{proof}
The petal distribution under $\mu_D$ is
$(d/(\mu(D)B),k)$-ranked-spread.  The $\mu_D$-mass of $\calA$ is
$\mu(\calA)/\mu(D)$.  Substitution into \cref{thm:ranked-hitting} cancels the
factor $\mu(D)$ and gives \eqref{eq:kernel-hitting-exact}.  The final assertion
follows from $H_r(\alpha)\le1/(\alpha-1)$.
\end{proof}

\subsection{Weighted kernel concentration}
\label{subsec:weighted-kernel}

The weighted analogue of \cref{cor:ranked-kernel-hitting} follows by applying the preceding theorem under the conditioned petal distribution.

\begin{corollary}[Weighted ranked kernel concentration]
\label{cor:kernel-concentration}
Under the notation of \cref{thm:ranked-extraction}, let
$\calA\subseteq D$, write $\tau=\mu(\calA)$, and let
$W\sim\Bin\left(U,\frac\alpha k\right)$, and $\frac\alpha k\le1$.
Define
$$
  X_{\calA}
  =
  \sum_{t\in\calA}
  \mu(t)\left(\frac\alpha k\right)^{-|P_t|}
  \one_{\{P_t\subseteq W\}}.
$$
Then
$$
  \E X_{\calA}=\tau,
  \qquad
  \operatorname{Var}(X_{\calA})
  \le
  \frac{\tau d}{B}H_q(\alpha),
$$
and, for every $\eta\in[0,1]$,
\begin{equation}
\label{eq:kernel-lower-tail}
  \Prb[X_{\calA}\le(1-\eta)\tau]
  \le
  \exp\left(
    -\frac{\eta^2\tau B}{2dH_q(\alpha)}
  \right).
\end{equation}
\end{corollary}

\begin{proof}
Apply \cref{thm:weighted-ranked} under $\mu_D$.  Its weighted score is
$X_{\calA}/\mu(D)$, its subfamily mass is $\tau/\mu(D)$, and its ranked
parameter is $d/(\mu(D)B)$.  The factors of $\mu(D)$ cancel in both the
variance and lower-tail estimates.
\end{proof}

\subsection{A robust daisy corollary}
\label{subsec:robust-daisies}

We now translate the preceding result into the terminology of Goldberg, Gur,
and Saraogi~\cite{goldberg2026robust}.  The definition is stated for
distributions on sets rather than indexed set systems, since this is the form
used in their work.

\begin{definition}[Robust daisy]
\label{def:robust-daisy}
Let $\nu$ be a probability distribution on a finite family of subsets of a
ground set $U$, let $p,\eta\in(0,1)$, and let $K\subseteq U$.  We say that
$\nu$ is a \emph{$(p,\eta)$-robust daisy with kernel $K$} if, for every
$\calA\subseteq\supp(\nu)$,
\begin{equation}
\label{eq:robust-daisy-definition}
  \Prb_{W\sim\Bin(U,p)}
  \bigl[\exists S\in\calA:\ S\setminus K\subseteq W\bigr]
  \ge 1-\eta^{\nu(\calA)}.
\end{equation}
Equivalently, the event in \eqref{eq:robust-daisy-definition} is that
$S\subseteq K\cup W$ for some $S\in\calA$.  The sets $S\setminus K$ are
called the \emph{petals}.
\end{definition}

Thus robustness is required not only for the full support, but for every
subfamily, with a failure probability that decays exponentially in its
probability mass.

\begin{corollary}[Improved robust daisy lemma]
\label{cor:robust-daisy}
Let $\mu$ be a probability distribution on subsets of an $n$-element ground
set $U$, suppose that every set in $\supp(\mu)$ has size at most $q$, and put
$d=\E_{S\sim\mu}|S|>0$.
Let $c>d$ be an integer and set $k=n^{1/c}$.  Then there exist
$\calD\subseteq\supp(\mu)$, a kernel $K\subseteq U$, an index
$j\in\{1,\ldots,c\}$, and
$B=\frac{n}{k^j}=n^{1-j/c}
$
such that
\begin{equation}
\label{eq:robust-daisy-extraction-parameters}
  \mu(\calD)\ge 1-\frac dc,
  \qquad
  |K|\le B\le n^{1-1/c}.
\end{equation}
Moreover, for every $\alpha>1$ for which
$p=\frac{\alpha}{k}=\alpha n^{-1/c}\le1$,
the conditioned distribution $\mu_{\calD}$ is a $(p,\eta)$-robust daisy with
kernel $K$, where
\begin{equation}
\label{eq:robust-daisy-eta}
  \eta=
  \exp\left(
    -\frac{(\alpha-1)\mu(\calD)B}{2d}
  \right).
\end{equation}
In particular, if only the worst-case bound $|S|\le q$ is retained and
$c>q$, then
\begin{equation}
\label{eq:robust-daisy-kernel-form}
  \mu(\calD)\ge1-\frac qc,
  \qquad
  \eta
  \le
  \exp\left(
    -\frac{(\alpha-1)(1-q/c)}{2q}|K|
  \right).
\end{equation}
\end{corollary}

\begin{proof}
Apply \cref{thm:ranked-extraction} to the indexed system whose indices are the
sets in $\supp(\mu)$ and whose attached set at index $S$ is $S$.  This gives
$\calD,K,j$, and $B$ satisfying
\eqref{eq:robust-daisy-extraction-parameters}.  Let
$\calA\subseteq\calD$.  By \cref{cor:ranked-kernel-hitting},
\begin{align*}
  \Prb_W\bigl[
    \nexists S\in\calA:\ S\setminus K\subseteq W
  \bigr]
  &\le
  \exp\left(
    -\frac{(\alpha-1)B}{2d}\mu(\calA)
  \right)\\
  &=
  \left[
    \exp\left(
      -\frac{(\alpha-1)\mu(\calD)B}{2d}
    \right)
  \right]^{\mu_{\calD}(\calA)}\\
  &=\eta^{\mu_{\calD}(\calA)}.
\end{align*}
This is precisely \cref{def:robust-daisy}.  Finally, $d\le q$,
$\mu(\calD)\ge1-d/c\ge1-q/c$, and $B\ge|K|$, which give
\eqref{eq:robust-daisy-kernel-form}.
\end{proof}

The preceding corollary is the promised robust-daisy deduction.  In the
normalization of Goldberg, Gur, and Saraogi
\cite[Lemma~5.1]{goldberg2026robust}, their stated extraction applies for
$\alpha>2q$ and gives the robustness parameter
$$
  \exp\left(
    -\frac{\alpha(1-q/c)}{8q^2\log|\calD|}|K|
  \right).
$$
By contrast, \cref{cor:robust-daisy} applies for every $\alpha>1$ and gives
the support-size-free exponent
$$
  \frac{(\alpha-1)(1-q/c)}{2q}|K|.
$$
More strongly, the natural exponent in \eqref{eq:robust-daisy-eta} is
proportional to the extraction scale $B$, which may be larger than the
kernel itself.  The improvement is exactly where ranked spreadness is used.

\section{Sample-based simulation}
\label{sec:simulation}

The two reductions use the same extracted kernel but exploit it in different
ways.  We therefore prove them separately.

\subsection{The one-sided reduction}
\label{subsec:one-sided-simulation}

\begin{proof}[Proof of \cref{thm:one-sided-main}]
Represent $T$ as a probability distribution $\mu$ on indexed deterministic
local tests $t=(Q_t,\varphi_t)$.  Apply \cref{thm:ranked-extraction} with the
chosen integer $c$.  We obtain a retained family $D\subseteq\calT$, a kernel
$K$, and a scale $B$ such that
\begin{equation}
\label{eq:one-sided-extraction}
  \mu(D)\ge1-\frac dc,
  \qquad
  |K|\le B\le n^{1-1/c}.
\end{equation}
Let $W\sim\Bin([n],p)$ with
$p=\alpha_{\mathrm{hit}}n^{-1/c}$.  For sufficiently large $n$, $p\le1$.

Given $(W,x|_W)$, process every hypothetical assignment $a\in\Sigma^K$.
Call $a$ \emph{eliminated} if there is a test $t\in D$ such that
$P_t:=Q_t\setminus K\subseteq W$ and $t$ rejects the local word obtained by
combining $a$ on $Q_t\cap K$ with the observed values on $P_t$.  Reject if and
only if every assignment is eliminated.

If $x\in\Pi$, the true assignment $a_*=x|_K$ is never eliminated, since every
exposed view is an actual view of $x$ and the original tester is one-sided.
Thus completeness is perfect.

Now suppose that $x$ is $\eps$-far from $\Pi$.  For $a\in\Sigma^K$, define
$x^{(a)}$ by replacing $x|_K$ with $a$.  Since
$|K|\le n^{1-1/c}\le\eps n/2$ for all sufficiently large $n$, every
$x^{(a)}$ is $\eps/2$-far from $\Pi$.  Hence its rejecting family
$\calR_a=\{t\in\calT:\varphi_t(x^{(a)}|_{Q_t})=1\}
$
has mass at least $\delta$, and therefore
\begin{equation}
\label{eq:one-sided-retained-mass}
  \mu(D\cap\calR_a)
  \ge\delta-\frac dc
  =\beta.
\end{equation}
By \cref{cor:ranked-kernel-hitting},
$$
  \Prb[a\text{ is not eliminated}]
  \le
  \exp\left(-\frac{(\alpha_{\mathrm{hit}}-1)\beta B}{2d}\right).
$$
There are at most $|\Sigma|^{|K|}\le\exp(B\log|\Sigma|)$ assignments.  A
union bound and the definition of $\alpha_{\mathrm{hit}}$ give
$$
  \Prb[\text{the new tester accepts }x]
  \le
  \exp\left(B\log|\Sigma|
       -\frac{(\alpha_{\mathrm{hit}}-1)\beta B}{2d}\right)
  =e^{-B\log3}
  \le\frac13.
$$
Finally, $\E|W|=pn=\alpha_{\mathrm{hit}}n^{1-1/c}$.
\end{proof}

\subsection{The two-sided reduction}
\label{subsec:two-sided-simulation}

The Boolean elimination rule cannot be used when $\gamma>0$: the true kernel
assignment may itself expose rejecting tests.  Instead, for each kernel
assignment we estimate the retained rejection mass and compare it with a
threshold lying strictly between the completeness and soundness regimes.

\begin{proof}[Proof of \cref{thm:two-sided-main}]
Represent $T$ as a distribution $\mu$ on indexed deterministic local tests
$t=(Q_t,\varphi_t)$ and apply \cref{thm:ranked-extraction}.  We obtain
$D,K,B$ satisfying
\begin{equation}
\label{eq:two-sided-extraction}
  \mu(D)\ge1-\frac dc,
  \qquad
  |K|\le B\le n^{1-1/c}.
\end{equation}
Let $p=\alpha_{\mathrm{wt}}n^{-1/c}$,
$W\sim\Bin([n],p)$, and
$\theta={\beta+\gamma}/{2}$.
For sufficiently large $n$, $p\le1$.

For $a\in\Sigma^K$, let $x^{(a)}$ agree with $a$ on $K$ and with $x$ outside
$K$, and let
$$
  \calR_a
  =\{t\in\calT:\varphi_t(x^{(a)}|_{Q_t})=1\}.
$$
Define the observable weighted rejection score
\begin{equation}
\label{eq:two-sided-score}
  Z_a
  =
  \sum_{t\in D\cap\calR_a}
  \mu(t)p^{-|P_t|}\one_{\{P_t\subseteq W\}},
  \qquad P_t=Q_t\setminus K.
\end{equation}
When $P_t\subseteq W$, the predicate value in the summand is determined by
the sample together with $a$.  The tester accepts if and only if
\begin{equation}
\label{eq:two-sided-decision}
  \min_{a\in\Sigma^K} Z_a\le\theta.
\end{equation}

\emph{Completeness.}
Suppose $x\in\Pi$ and take $a_*=x|_K$.  Put
$\tau_*=\mu(D\cap\calR_{a_*})$.  Then $\tau_*\le r_T(x)\le\gamma$.
By \cref{cor:kernel-concentration},
$$
  \E Z_{a_*}=\tau_*,
  \qquad
  \operatorname{Var}(Z_{a_*})
  \le\frac{\tau_*d}{B}H_q(\alpha_{\mathrm{wt}})
  \le\frac{\gamma d}{B}H_q(\alpha_{\mathrm{wt}}).
$$
Since $\theta-\tau_*\ge g/2$, Chebyshev's inequality and
$H_q(\alpha)\le(\alpha-1)^{-1}$ give
\begin{align*}
  \Prb[\text{the new tester rejects }x]
  &\le \Prb[Z_{a_*}>\theta]\\
  &\le \frac{4\gamma dH_q(\alpha_{\mathrm{wt}})}{Bg^2}\\
  &\le \frac{4\gamma d}{(\alpha_{\mathrm{wt}}-1)g^2}
  \le\frac13.
\end{align*}
The case $\gamma=0$ is included: then $\tau_*=0$ and $Z_{a_*}=0$ almost
surely.

\emph{Soundness.}
Suppose $x$ is $\eps$-far from $\Pi$.  As above, every completion $x^{(a)}$
is $\eps/2$-far for sufficiently large $n$.  Therefore
\begin{equation}
\label{eq:two-sided-retained-mass}
  \tau_a:=\mu(D\cap\calR_a)
  \ge\delta-\frac dc
  =\beta.
\end{equation}
Since $\theta<\tau_a$, apply \cref{cor:kernel-concentration} with relative
deviation $1-\theta/\tau_a$.  Using $\tau_a\le1$ and
$\tau_a-\theta\ge g/2$, we obtain
\begin{align}
  \Prb[Z_a\le\theta]
  &\le
  \exp\left(-\frac{B(\tau_a-\theta)^2}
                  {2dH_q(\alpha_{\mathrm{wt}})\tau_a}\right)\notag\\
  &\le
  \exp\left(-\frac{Bg^2}{8dH_q(\alpha_{\mathrm{wt}})}\right)\notag\\
  &\le
  \exp\left(-\frac{(\alpha_{\mathrm{wt}}-1)Bg^2}{8d}\right).
\label{eq:two-sided-one-assignment}
\end{align}
A union bound over at most $|\Sigma|^{|K|}\le e^{B\log|\Sigma|}$ assignments
and the definition of $\alpha_{\mathrm{wt}}$ give
$$
  \Prb[\text{the new tester accepts }x]
  \le
  \exp\left(B\log|\Sigma|
       -\frac{(\alpha_{\mathrm{wt}}-1)Bg^2}{8d}\right)
  \le e^{-B\log3}
  \le\frac13.
$$
Finally, $\E|W|=pn=\alpha_{\mathrm{wt}}n^{1-1/c}$.
\end{proof}

\begin{proof}[Proof of \cref{cor:q-formss}]
Use $d\le q$.  In the one-sided case choose
$c=\lfloor q/\delta\rfloor+1$.  In the two-sided case choose
$c=\lfloor q/(\delta-\gamma)\rfloor+1$.  For
$\gamma=1/3$ and $\delta=2/3$, the latter gives $c=3q+1$.
\end{proof}

Note that setting $\gamma=0$ in \cref{thm:two-sided-main} gives the same exponent as
\cref{thm:one-sided-main}, but the weighted proof requires
$\alpha_{\mathrm{wt}}-1=O(d\beta^{-2}\log|\Sigma|)$, whereas the hitting
proof requires only
$\alpha_{\mathrm{hit}}-1=O(d\beta^{-1}\log|\Sigma|)$.  This is not a cosmetic
constant loss: it reflects the distinction between merely finding one
rejecting witness and estimating rejection mass to additive accuracy
comparable with the gap.

\printbibliography

\end{document}